\documentclass[preprint,11pt,oneside,onecolumn]{elsarticle}
\usepackage{array}
\usepackage{color}
\usepackage{mathrsfs}
\biboptions{numbers,sort&compress}
\usepackage{hyperref}
\usepackage{amssymb}
\usepackage{amsmath}
\usepackage{amsthm}
\allowdisplaybreaks
\usepackage{graphicx}
\usepackage{tikz}

\usetikzlibrary{shapes.arrows,chains,positioning}
\newtheorem{Theorem}{Theorem}[section]
\newtheorem{Lemma}[Theorem]{Lemma}
\newtheorem{Definition}[Theorem]{Definition}
\newtheorem{Proposition}[Theorem]{Proposition}

\begin{document}

\begin{frontmatter}


\title{Multivariate Dynamical Sampling in $l^2(\mathbb{Z}^2)$ and Shift-Invariant Spaces Associated with Linear Canonical Transform
\tnoteref{label1}}
\tnotetext[label1]{This work was partially supported by the National Natural Science Foundation of China(Grant Nos. 11801256 and 12061044).}

\author{Haiye Huo}
\ead{hyhuo@ncu.edu.cn}

\address{Department of Mathematics, School of Science, Nanchang University, Nanchang~330031, Jiangxi, China}

\author{Li Xiao\corref{corau}}
\cortext[corau]{Corresponding author}
\ead{xiaoli11@ustc.edu.cn}

\address{School of Information Science and Technology, University of Science and Technology of China, Hefei~230052, Anhui, China}

\begin{abstract}
In this paper, we investigate the multivariate dynamical sampling problem in $l^2(\mathbb{Z}^2)$ associated with the two-dimensional discrete time non-separable linear canonical transform (2D-DT-NS-LCT) and
shift-invariant spaces associated with the two-dimensional non-separable linear canonical transform (2D-NS-LCT), respectively. Specifically, we derive a sufficient and necessary condition under which a sequence in $l^2(\mathbb{Z}^2)$ (or a function in a shift-invariant space) can be stably recovered from its dynamical sampling measurements associated with the 2D-DT-NS-LCT (or the 2D-NS-LCT). We also present a simple example to elucidate our main results.
\end{abstract}

\end{frontmatter}


\section{Introduction}\label{sec1}

The linear canonical transform (LCT) is a general class of linear integral transformations with three free parameters, which includes many well-known
linear transforms as its special cases, such as Fourier transform (FT), fractional FT, scaling operation, and Fresnel transform \cite{book1,KI2012,FL2016,Huo2018A,WRL2009,ZZhang2015,XFH2021}. Therefore,
there has been growing interest in studying the LCT and its properties pertaining to applications across
the fields of signal processing and optics \cite{OZK2001,QLL2013,SO1998,xxx1}.

Signal sampling is a fundamental concept in digital signal processing, as it provides a bridge between continuous- and discrete-time signals. A variety of sampling theorems in the traditional FT domain have been generalized to the LCT domain in the broad sense that signals which are non-bandlimited in the FT domain may be bandlimited in the LCT domain \cite{Stern2006A,TLW2008,HS2015,SLS2012a,Huo2019N,XS2013,xiao11}.
In recent years, dynamical sampling has attracted empirical attention in the scientific community, which is a new way of signal sampling, compared with classical sampling techniques, and has potential applications to wireless sensor networks in the health, environment, and precision agriculture industries \cite{AT2014,ACC2021,CMP2020,CHR2018,jia1,jia2}. More specifically, the dynamical sampling refers to not only the signal $f$ that is sampled but also its various states at different times $\{t_1,t_2,\cdots,t_N\}$. It is to consider the problem of spatiotemporal sampling with which an initial state $f$ of an evolution process $f_t=A_tf$ is recovered from a combined set of coarse spatial samples $\{f(X), f_{t_1}(X), \cdots, f_{t_N}(X)\}$ of $f$ on $X\subset\Omega$ at varying time levels $\{t_1,t_2,\cdots,t_N\}$, where $\Omega$ is the domain of $f$. Most work of dynamical sampling aimed to derive some sufficient and necessary conditions on $A$ and $\Omega$ such that $f$ can be stably recovered from its dynamical sampling measurements $\{f(X),f_{t_1}(X),\cdots,f_{t_N}(X)\}$.

Of note, Aldroubi and his collaborators have intensively studied the dynamical sampling of signals in finite dimensional spaces \cite{ADK2013}, shift-invariant spaces
\cite{ADK2015,AAD2013}, and infinite dimensional spaces (e.g., $l^2(\mathbb{Z})$) \cite{ACM2017}.
Zhang et al \cite{ZLL2017} investigated the periodic nonuniform dynamical sampling in $l^2(\mathbb{Z})$ and shift-invariant spaces associated with the FT. More recently, Liu et al \cite{LWL2021} extended the dynamical sampling results associated with the FT to the special affine Fourier transform (also known as offset linear canonical transform) domain in a broader sense, i.e., signals are recovered from their dynamical sampling measurements associated with the special affine Fourier transform. However, all the mentioned dynamical sampling results above focus on one-dimensional signals. Motivated by applications in multi-dimensional signal systems, Zhang et al \cite{ZL2019} considered the multivariate dynamical sampling in $l^2(\mathbb{Z}^d)$ and
shift-invariant spaces associated with the $d$-dimensional FT.

To the best of our knowledge, there have been no studies on the multivariate dynamical sampling associated with the multi-dimensional LCT. In this paper, we therefore propose to investigate
the multivariate dynamical sampling of signals in $l^2(\mathbb{Z}^2)$ associated with the two-dimensional discrete time non-separable LCT (2D-DT-NS-LCT, defined below) and shift-invariant spaces associated with the two-dimensional non-separable LCT (2D-NS-LCT, defined below), respectively.
We derive a sufficient and necessary condition under which a sequence in $l^2(\mathbb{Z}^2)$ (or a function in a shift-invariant space $V(\phi)$ that is generated by $\phi\in L^2(\mathbb{R}^2)$) can be stably reconstructed from its dynamical sampling measurements associated with the 2D-DT-NS-LCT (or the 2D-NS-LCT).

The rest of this paper is organized as follows. In Section \ref{sec2}, we first introduce definitions of the 2D-NS-LCT, the 2D-DT-NS-LCT, and their related canonical convolution operators, and then obtain some important properties that will be utilized later. In Section \ref{sec3}, we consider the multivariate dynamical sampling in $l^2(\mathbb{Z}^2)$ (i.e., we derive a sufficient and necessary condition for sequences in $l^2(\mathbb{Z}^2)$ under which they can be recovered from their dynamical sampling measurements in a stable way) associated with the 2D-DT-NS-LCT, followed by the multivariate dynamical sampling in shift-invariant spaces associated with the 2D-NS-LCT in Section \ref{sec:shift}. In Section \ref{jjjia}, we present an example to elucidate our main results, and conclude this paper in Section \ref{sec5}.

\section{Preliminaries}\label{sec2}
The two-dimensional non-separable LCT (2D-NS-LCT) with parameter $\mathcal{M}\triangleq[A, B; C, D]$
of a signal $f\in L^2(\mathbb{R}^2)$ is defined by \cite{AB2006,KHO2010,PH2016,ZHS2014,ZWL2019}
\begin{equation}\label{def:LCT:L0}
(L_{\mathcal{M}}f)(\xi)=
    \frac{1}{\sqrt{\det{(iB)}}}\int_{\mathbb{R}^2} f(u)e^{i\pi u^TB^{-1}Au-2i\pi u^{T}B^{-1}\xi+i\pi \xi^{T}DB^{-1}\xi}{\rm{d}}u,
\end{equation}
where $u, \xi\in\mathbb{R}^2$ are two real column vectors, the superscript $(\cdot)^{T}$ denotes the transpose of a vector or matrix, $\det(\cdot)$ stands for the determinant of a matrix, and $A, B, C, D\in\mathbb{R}^{2\times2}$ are $2\times 2$ real matrices with $B$ being non-singular. The matrix $\mathcal{M}$ is real and symplectic so that the following equations hold:
\begin{equation}\label{ssads}
AB^T=BA^T,\; CD^T=DC^T, \;AD^T-BC^T=I,
\end{equation}
or
\begin{equation}
A^TC=C^TA,\; B^TD=D^TB, \;A^TD-C^TB=I.
\end{equation}
Here, $I$ stands for the $2\times 2$ identity matrix. From the perspective of group theory, the 2D-NS-LCT forms a symplectic group $S_p(4,\mathbb{R})$ with
ten parameters. We refer the reader to \cite{OZK1995,DWY2019} for more details about the 2D-NS-LCT.

To begin with, we introduce a new canonical convolution operator associated with the 2D-NS-LCT. Let $\lambda_{\mathcal{M}}(t)=e^{i\pi t^{T}B^{-1}At}$ be the
chirp-modulation function, and define

\begin{eqnarray*}
\overrightarrow{f}(t)&:=&\lambda_{\mathcal{M}}(t)f(t)=e^{i\pi t^{T}B^{-1}At}f(t),\\
\overleftarrow{f}(t)&:=&\overline{\lambda}_{\mathcal{M}}(t)f(t)=e^{-i\pi t^{T}B^{-1}At}f(t),
\end{eqnarray*}
where $\overline{z}$ means the conjugate of $z$.

\begin{Definition}\label{ C-Can-Conv}
Define the canonical convolution operator $\star_{c}$ of two functions $f,\;g\in L^2(\mathbb{R}^2)$ associated with the 2D-NS-LCT as
\begin{equation}\label{def:Can-Conv1}
(f\star_{c} g)(t)=\frac{\overline{\lambda}_{\mathcal{M}}(t)}{\sqrt{\det{(iB)}}}(\overrightarrow{f}\ast \overrightarrow{g})(t),
\end{equation}
where $\ast$ denotes the traditional convolution operator, i.e.,
\begin{equation}\label{def:Can-Conv2}
(f\ast g)(t)=\int_{\mathbb{R}^2}f(t-x)g(x){\rm{d}}x.
\end{equation}
\end{Definition}

We then have the following lemma about the canonical convolution operator $\star_{c}$ of two functions in the 2D-NS-LCT domain.
\begin{Lemma}\label{lem:Conv1}
For two functions $f,\;g\in L^2(\mathbb{R}^2)$, let $h(t)=(f\star_{c}g)(t)$. Then, the 2D-NS-LCT $(L_\mathcal{M}h)(\xi)$ of $h$ satisfies
\begin{equation}\label{eq:Conv1}
(L_\mathcal{M}h)(\xi)=\overline{\eta}_{\mathcal{M}}(\xi)(L_\mathcal{M}f)(\xi)(L_\mathcal{M}g)(\xi),
\end{equation}
where $\eta_{\mathcal{M}}(\xi)=e^{i\pi\xi^TDB^{-1}\xi}$.
\end{Lemma}
\begin{proof}
According to (\ref{def:LCT:L0}), (\ref{def:Can-Conv1}) and (\ref{def:Can-Conv2}), we have
\begin{align*}
&(L_\mathcal{M}h)(\xi)\\
=&\frac{1}{\sqrt{\det{(iB)}}}\int_{\mathbb{R}^2} h(u)e^{i\pi u^TB^{-1}Au-2i\pi u^{T}B^{-1}\xi+i\pi \xi^{T}DB^{-1}\xi}{\rm{d}}u\\
=&\frac{1}{\sqrt{\det{(iB)}}}\int_{\mathbb{R}^2} \frac{\overline{\lambda}_{\mathcal{M}}(u)}{\sqrt{\det{(iB)}}}
\bigg(\int_{\mathbb{R}^2}\overrightarrow{f}(u-x)\overrightarrow{g}(x){\rm{d}}x\bigg)\\
 &\quad\times e^{i\pi u^TB^{-1}Au-2i\pi u^{T}B^{-1}\xi+i\pi \xi^{T}DB^{-1}\xi}{\rm{d}}u\\
=&\frac{1}{\det{(iB)}}\int_{\mathbb{R}^2}\int_{\mathbb{R}^2}f(u-x)e^{i\pi(u-x)^TB^{-1}A(u-x)}g(x)e^{i\pi x^TB^{-1}Ax}\\
 &\quad\times e^{-2i\pi u^{T}B^{-1}\xi+i\pi \xi^{T}DB^{-1}\xi}{\rm{d}}u{\rm{d}}x\\
=&\frac{\overline{\eta}_{\mathcal{M}}(\xi)}{\sqrt{\det{(iB)}}}\int_{\mathbb{R}^2}f(u-x)e^{i\pi(u-x)^TB^{-1}A(u-x)-2i\pi (u-x)^{T}B^{-1}\xi+i\pi \xi^{T}
  DB^{-1}\xi}{\rm{d}}u\\
 &\quad\times\frac{1}{\sqrt{\det{(iB)}}}\int_{\mathbb{R}^2}g(x)e^{i\pi x^TB^{-1}Ax-2i\pi x^{T}B^{-1}\xi+i\pi \xi^{T}DB^{-1}\xi}{\rm{d}}x\\
=&\overline{\eta}_{\mathcal{M}}(\xi)(L_\mathcal{M}f)(\xi)(L_\mathcal{M}g)(\xi),
\end{align*}
which completes the proof.
\end{proof}

We next give definitions of the two-dimensional discrete time non-separable LCT (2D-DT-NS-LCT) of a sequence in $l^2(\mathbb{Z}^2)$, the canonical convolution of a
sequence in $l^2(\mathbb{Z}^2)$ and a function in $L^2(\mathbb{R}^2)$, and periodic
functions, respectively.

\begin{Definition}\label{def:DLCT}
Let $s={s(k)}$ be a sequence in $l^2(\mathbb{Z}^2)$, i.e., $\sum_{k\in \mathbb{Z}^2}|s(k)|^2<+\infty$. The 2D-DT-NS-LCT of $s$ is defined by
\begin{equation}\label{def:DLCT:L0}
(L_{\mathcal{M}}s)(\xi)=
    \frac{1}{\sqrt{\det{(iB)}}}\sum_{k\in \mathbb{Z}^2}s(k)e^{i\pi k^TB^{-1}Ak-2i\pi k^{T}B^{-1}\xi+i\pi \xi^{T}DB^{-1}\xi}.
\end{equation}
\end{Definition}

\begin{Definition}\label{Se-Conv}
The canonical convolution operator $\star_{sd}$ of a sequence $s\in l^2(\mathbb{Z}^2)$ and a function $\phi\in L^2(\mathbb{R}^2)$ is defined as
\begin{equation}\label{def:Se-Conv}
h(t)=(s\star_{sd}\phi)(t)=\frac{\overline{\lambda}_{\mathcal{M}}(t)}{\sqrt{\det{(iB)}}}\sum_{k\in \mathbb{Z}^2}
\lambda_{\mathcal{M}}(k)s(k)\lambda_{\mathcal{M}}(t-k)\phi(t-k).
\end{equation}
\end{Definition}

\begin{Definition}
A function $\phi(t)$ is said to be periodic with periodicity matrix $M\in \mathbb{R}^{2\times 2}$, if
\begin{equation}\label{def:peri}
\phi(t+Mn)=\phi(t),\;\;{\mbox{for all }}t\in\mathbb{R}^{2} \mbox{and any } n\in \mathbb{Z}^{2},
\end{equation}
where $M$ is non-singular.
\end{Definition}


In the following, we derive a canonical convolution theorem of
a sequence in $l^2(\mathbb{Z}^2)$ and a function in $L^2(\mathbb{R}^2)$.

\begin{Lemma}\label{lem:Se-Conv}
For $s\in l^2(\mathbb{Z}^2)$, $\phi\in L^2(\mathbb{R}^2)$, let $h(t)=(s\star_{sd}\phi)(t)$. Then,
\begin{equation}\label{eq:Conv}
(L_\mathcal{M}h)(\xi)=\overline{\eta}_{\mathcal{M}}(\xi)(L_\mathcal{M}s)(\xi)(L_\mathcal{M}\phi)(\xi).
\end{equation}
Moreover, $|(L_{\mathcal{M}}s)(\xi)|$ is periodic with periodicity matrix $B$.
\end{Lemma}
\begin{proof}
According to (\ref{def:LCT:L0}), (\ref{def:DLCT:L0}) and (\ref{def:Se-Conv}), we have
\begin{align*}
&(L_\mathcal{M}h)(\xi)\\
=&\frac{1}{\sqrt{\det{(iB)}}}\int_{\mathbb{R}^2}(s\star_{sd}\phi)(t)e^{i\pi t^TB^{-1}At-2i\pi t^{T}B^{-1}\xi+i\pi \xi^{T}DB^{-1}\xi}{\rm{d}}t\\
=&\frac{1}{\det{(iB)}}\sum_{k\in \mathbb{Z}^2}\lambda_{\mathcal{M}}(k)s(k)\bigg(\int_{\mathbb{R}^2}\overline{\lambda}_{\mathcal{M}}(t)\lambda_{\mathcal{M}}(t-k)\phi(t-k)\\
&\quad \times e^{i\pi t^{T}B^{-1}At-2i\pi t^{T}B^{-1}\xi+i\pi \xi^{T}DB^{-1}\xi}{\rm{d}}t\bigg)\\
=&\frac{1}{\sqrt{\det{(iB)}}}\sum_{k\in\mathbb{Z}^2}s(k)e^{i\pi k^TB^{-1}Ak-2i\pi k^TB^{-1}\xi} \bigg(\frac{1}{\sqrt{\det{(iB)}}}\int_{\mathbb{R}^2}\phi(t)\\
&\quad \times e^{i\pi t^{T}B^{-1}At-2i\pi t^{T}B^{-1}\xi+i\pi \xi^{T}DB^{-1}\xi}{\rm{d}}t\bigg)\\
=&\overline{\eta}_{\mathcal{M}}(\xi)(L_\mathcal{M}s)(\xi)(L_\mathcal{M}\phi)(\xi).
\end{align*}
Furthermore, for all $\xi\in \mathbb{Z}^{2}$ and any $l\in \mathbb{Z}^{2}$, we get
\begin{align}\label{peri:1}
&(L_{\mathcal{M}}s)(\xi+Bl)\nonumber\\
=&\frac{1}{\sqrt{\det{(iB)}}}\sum_{k\in \mathbb{Z}^2}s(k)e^{i\pi k^TB^{-1}Ak-2i\pi k^{T}B^{-1}(\xi+Bl)+i\pi (\xi+Bl)^{T}DB^{-1}(\xi+Bl)}\nonumber\\
=&\frac{1}{\sqrt{\det{(iB)}}}\sum_{k\in \mathbb{Z}^2}s(k)e^{i\pi k^TB^{-1}Ak-2i\pi k^{T}B^{-1}\xi+i\pi \xi^{T}DB^{-1}\xi} e^{-2i\pi k^{T}l}\nonumber\\
 &\times e^{i\pi(Bl)^{T}DB^{-1}\xi+i\pi \xi^{T}Dl+i\pi (Bl)^{T}Dl}\nonumber\\
 =&(L_{\mathcal{M}}s)(\xi)e^{i\pi(Bl)^{T}DB^{-1}\xi+i\pi \xi^{T}Dl+i\pi (Bl)^{T}Dl},
\end{align}
where we use $e^{-2i\pi k^{T}l}=1$ in the last step.
Hence,
\[
|(L_{\mathcal{M}}s)(\xi+Bl)|=|(L_{\mathcal{M}}s)(\xi)|.
\]
This completes the proof.
\end{proof}

Similarly, we introduce the definition and property of the canonical convolution of two sequences in $l^2(\mathbb{Z}^2)$ associated with the 2D-DT-NS-LCT as follows.
\begin{Definition}\label{D-Conv}
Let $s={s(k)}\in l^2(\mathbb{Z}^2)$ and $c={c(k)}\in l^2(\mathbb{Z}^2)$. The canonical convolution operator $\star_{d}$ of two sequences $s$ and $c$ is defined by
\begin{equation}\label{def:D-Conv}
h(l)=(s\star_{d}c)(l)=\frac{\overline{\lambda}_{\mathcal{M}}(l)}{\sqrt{\det{(iB)}}}\sum_{k\in \mathbb{Z}^2}
\lambda_{\mathcal{M}}(k)s(k)\lambda_{\mathcal{M}}(l-k)c(l-k).
\end{equation}
\end{Definition}

\begin{Lemma}\label{lem:D-Conv}
For $s,\;c\in l^2(\mathbb{Z}^2)$, let $h(l)=(s\star_{d}c)(l)$. Then, we have
\begin{equation}\label{eq:D-Conv}
(L_\mathcal{M}h)(\xi)=\overline{\eta}_{\mathcal{M}}(\xi)(L_\mathcal{M}s)(\xi)(L_\mathcal{M}c)(\xi).
\end{equation}
\end{Lemma}

\begin{proof}
By (\ref{def:DLCT:L0}) and (\ref{def:D-Conv}), we have
\begin{align*}
&(L_\mathcal{M}h)(\xi)\\
=&\frac{1}{\sqrt{\det{(iB)}}}\sum_{l\in \mathbb{Z}^2}(s\star_{d}c)(l)e^{i\pi l^TB^{-1}Al-2i\pi l^{T}B^{-1}\xi+i\pi \xi^{T}DB^{-1}\xi}\\
=&\frac{1}{\det{(iB)}}\sum_{l\in \mathbb{Z}^2}\overline{\lambda}_{\mathcal{M}}(l)\bigg(\sum_{k\in \mathbb{Z}^2}
\lambda_{\mathcal{M}}(k)s(k)\lambda_{\mathcal{M}}(l-k)c(l-k)\bigg)\\
&\quad\times e^{i\pi l^TB^{-1}Al-2i\pi l^{T}B^{-1}\xi+i\pi\xi^{T}DB^{-1}\xi}\\
=&\frac{1}{\det{(iB)}}\sum_{k\in \mathbb{Z}^2}\lambda_{\mathcal{M}}(k)s(k)\bigg(\sum_{l\in \mathbb{Z}^2}\overline{\lambda}_{\mathcal{M}}(l)
\lambda_{\mathcal{M}}(l-k)c(l-k)\\
&\quad\times e^{i\pi l^TB^{-1}Al-2i\pi l^{T}B^{-1}\xi+i\pi\xi^{T}DB^{-1}\xi}\bigg)\\
=&\frac{1}{\sqrt{\det{(iB)}}}\sum_{k\in \mathbb{Z}^2}s(k)e^{i\pi k^TB^{-1}Ak-2i\pi k^{T}B^{-1}\xi}\bigg(\frac{1}{\sqrt{\det{(iB)}}}\sum_{l\in \mathbb{Z}^2}c(l)\\
&\quad\times e^{i\pi l^TB^{-1}Al-2i\pi l^{T}B^{-1}\xi+i\pi\xi^{T}DB^{-1}\xi}\bigg)\\
=&\overline{\eta}_{\mathcal{M}}(\xi)(L_\mathcal{M}s)(\xi)(L_\mathcal{M}c)(\xi),
\end{align*}
which completes the proof.
\end{proof}

On the basis of the above lemmas, we propose to explore the multivariate dynamical sampling in the $l^2(\mathbb{Z}^2)$ space associated with the 2D-DT-NS-LCT and shift-invariant spaces associated with the 2D-NS-LCT, respectively, in what follows.

\section{Multivariate dynamical sampling in $l^2(\mathbb{Z}^2)$ associated with 2D-DT-NS-LCT}\label{sec3}
In this section, we focus on the multivariate dynamical sampling in $l^2(\mathbb{Z}^2)$ (i.e., a sequence space) associated with the 2D-DT-NS-LCT. Specifically, we derive a necessary and sufficient condition under which signals (or sequences) in $l^2(\mathbb{Z}^2)$ can be stably reconstructed from their dynamical sampling values associated with the 2D-DT-NS-LCT.

Let $M$ be a $2\times 2$ non-singular real matrix (not necessarily integer matrix) with $m=|\det(M)|$, and $M\mathbb{Z}^2\triangleq\{Mn\; |\,n\in\mathbb{Z}^2\}$ be a lattice generated by $M$ \cite{DWY2019,XXW2020}.
Let us define $\mathbb{T}^2\triangleq\{[x_1,x_2]^T\; |\,x_1\in[0,1)\mbox{ and } x_2\in[0,1)\}\subset\mathbb{R}^2$.
The fundamental parallelepiped of $M\mathbb{Z}^2$ is defined as the region
\[
\mathcal{G}(M)\triangleq\left\{Mx\; |\,x\in\mathbb{T}^2\right\}.
\]
One can see that $\mathcal{G}(M)$ and its shifted copies (called the other lattice cells) constitute the whole real vector space $\mathbb{R}^2$, i.e.,
\begin{equation}\label{ssheng}
\bigcup_{n\in\mathbb{Z}^2}\left\{M(x+n)\; |\,x\in\mathbb{T}^2\right\}=\mathbb{R}^2.
\end{equation}

When $M$ is further an integer matrix, we define
\[
\mathcal{N}(M)\triangleq\left\{k\; |\,k=Mx, x\in\mathbb{T}^2, \mbox{and }k\in\mathbb{Z}^2\right\}.
\]
The number of elements in $\mathcal{N}(M)$ equals $m$ \cite{XXW2020}. Without loss of generality, let $\gamma_0=[0, 0]^{T}, \gamma_1, \cdots, \gamma_{m-1}$ be the $m$ distinct elements in $\mathcal{N}(M)$. It is clear that $\gamma_k+M\mathbb{Z}^2 (k=0,1,\cdots,m-1)$ constitute the whole integer vector space $\mathbb{Z}^2$, i.e., $\bigcup_{k=0}^{m-1}\{\gamma_k+M\mathbb{Z}^2\}=\mathbb{Z}^2$. Analogously, let $M^T\mathbb{Z}^2\triangleq\{M^Tn\; |\,n\in\mathbb{Z}^2\}$ be a lattice generated by $M^T$, and $\eta_0=[0, 0]^{T}, \eta_1, \cdots, \eta_{m-1}$ be the $m$ distinct elements in $\mathcal{N}(M^T)$. Obviously, $\bigcup_{j=0}^{m-1}\{\eta_j+M^T\mathbb{Z}^2\}=\mathbb{Z}^2$.

Let
\[
S_M:\;l^2(\mathbb{Z}^2)\rightarrow l^2(\mathbb{Z}^2)
\]
be the subsampling operator by some fixed dilation matrix $M$ (i.e., a $2\times 2$ non-singular integer matrix), that is to say, for any $c\in l^2(\mathbb{Z}^2),$ and $k\in \mathbb{Z}^2$,
\[
(S_M c)(k)=c(M^Tk).
\]

In the beginning, we introduce some propositions and lemmas, which are useful for tackling the multivariate dynamical sampling problem in
$l^2(\mathbb{Z}^2)$ associated with the 2D-DT-NS-LCT.

\begin{Proposition}[\cite{XXW2020}]\label{ortho}
As stated above, let $\{\gamma_k\}_{k=0}^{m-1}\in\mathbb{Z}^2$ and $\{\eta_j\}_{j=0}^{m-1}\in\mathbb{Z}^2$ be the $m$ distinct elements of
$\mathcal{N}(M)$ and $\mathcal{N}(M^T)$, respectively, with $\gamma_0=\eta_0=[0, 0]^{T}$. Then,
\begin{equation}\label{ortho1}
\sum_{k=0}^{m-1}e^{-i2\pi\eta_j^TM^{-1}\gamma_k}=m\delta_j,\;\;j=0, 1, \cdots, m-1,
\end{equation}
where $m=|\det(M)|=|\det(M^T)|$, and
\begin{equation*}
\delta_j=
\begin{cases}
1,& j=0,\\
0,& j\ne 0.
\end{cases}
\end{equation*}
\end{Proposition}

Based on Proposition \ref{ortho}, we next obtain the Poisson summation formula in the 2D-DT-NS-LCT domain as follows.
\begin{Lemma}\label{lem:Poi}
Let a sequence $c\in l^2(\mathbb{Z}^2)$. Then, we have
\begin{align}
&\frac{m}{\sqrt{\det(iB)}}\{\mathcal{F}[S_M (c\lambda_{\mathcal{M}})]\}\big(\xi\big)\nonumber\\
=&\sum_{k=0}^{m-1}\overline{\eta}_\mathcal{M}[BM^{-1}(\xi+\gamma_k)](L_{\mathcal{M}}c)[BM^{-1}(\xi+\gamma_k)],\label{eq:Poi}
\end{align}
where $\mathcal{F}$ stands for the discrete time Fourier transform on $\mathbb{T}^2$, i.e., for any $c\in l^2(\mathbb{Z}^2)$,
\[
(\mathcal{F}c)(\xi)=\sum_{k\in \mathbb{Z}^2}c(k)e^{-2i\pi k^T\xi}, \;\,\xi\in\mathbb{T}^2.
\]
\end{Lemma}

\begin{proof}
Combining (\ref{def:DLCT:L0}) and (\ref{ortho1}), we have
\begin{align}
&\sum_{k=0}^{m-1}\overline{\eta}_\mathcal{M}[BM^{-1}(\xi+\gamma_k)](L_{\mathcal{M}}c)[BM^{-1}(\xi+\gamma_k)]\nonumber\\
=&\frac{1}{\sqrt{\det{(iB)}}}\sum_{k=0}^{m-1}e^{-i\pi[BM^{-1}(\xi+\gamma_k)]^{T}DM^{-1}(\xi+\gamma_k)}\sum_{n\in \mathbb{Z}^2}c(n)e^{i\pi n^TB^{-1}An}\nonumber\\
&\quad\times e^{-2i\pi n^{T}B^{-1}[BM^{-1}(\xi+\gamma_k)]}e^{i\pi [BM^{-1}(\xi+\gamma_k)]^{T}DB^{-1}[BM^{-1}(\xi+\gamma_k)]}\nonumber\\
=&\frac{1}{\sqrt{\det{(iB)}}}\sum_{k=0}^{m-1}\sum_{n\in \mathbb{Z}^2}c(n)e^{i\pi n^TB^{-1}An}e^{-2i\pi n^{T}M^{-1}(\xi+\gamma_k)}\nonumber\\
=&\frac{1}{\sqrt{\det{(iB)}}}\sum_{n\in \mathbb{Z}^2}c(n)e^{i\pi n^TB^{-1}An}e^{-2i\pi n^{T}M^{-1}\xi}\sum_{k=0}^{m-1}e^{-2i\pi n^{T}M^{-1}\gamma_k}\nonumber\\
=&\frac{1}{\sqrt{\det{(iB)}}}\sum_{n^{\prime}\in\mathbb{Z}^2,n=M^Tn^{\prime}}c(n)e^{i\pi n^TB^{-1}An}e^{-2i\pi n^{T}M^{-1}\xi}\sum_{k=0}^{m-1}e^{-2i\pi n^{T}M^{-1}\gamma_k}\nonumber\\
&\quad+\frac{1}{\sqrt{\det{(iB)}}}\sum_{j=1}^{m-1}\sum_{n^{\prime}\in\mathbb{Z}^2,n=M^Tn^{\prime}+\eta_j}c(n)e^{i\pi n^TB^{-1}An}e^{-2i\pi n^{T}M^{-1}\xi}\nonumber\\
&\quad\quad\times\sum_{k=0}^{m-1}e^{-2i\pi n^{T}M^{-1}\gamma_k}\nonumber\\
=&\frac{m}{\sqrt{\det{(iB)}}}\sum_{n^{\prime}\in\mathbb{Z}^2}c(M^Tn^{\prime})e^{i\pi(M^Tn^{\prime})^TB^{-1}A(M^Tn^{\prime})}e^{-2i\pi [M^Tn^{\prime}]^{T}M^{-1}\xi}\nonumber\\
=&\frac{m}{\sqrt{\det{(iB)}}}\{\mathcal{F}[S_M (c\lambda_{\mathcal{M}})]\}\big(\xi\big).\label{eq:Poi1}
\end{align}
This completes the proof.
\end{proof}

Let $a\in l^2(\mathbb{Z}^2)$ be the kernel of an evolution operator, $a^j=\underbrace{a\star_{d}\cdots\star_{d} a}_{j}, j=1,2,\cdots,m-1$, and $y_0(k)=c(M^Tk), y_j(k)=(a^j\star_{d} c)(M^Tk), j=1,2,\cdots,m-1$. Then, based on Lemma~\ref{lem:Poi} and Lemma~\ref{lem:D-Conv}, we have the following theorem.

\begin{Theorem}\label{thm:dynam}
Let $a$ be a sequence with $(L_\mathcal{M}a)(\xi)\in L^{\infty}(\mathbb{R}^2)$. Suppose that $e_{j}^{k}=\overline{\eta}_{\mathcal{M}}^{j+1}[BM^{-1}(\xi+\gamma_k)], k=0,1,\cdots,m-1$,
and $\mathcal{A}_{\mathcal{M}}(\xi)$ is denoted as
\begin{equation}
\tiny{\mathcal{A}_{\mathcal{M}}(\xi)=\left[
                     \begin{array}{cccc}
                       e_0^0 & e_0^1 & \cdots & e_0^{m-1} \\
                       e_1^0(L_\mathcal{M}a)(BM^{-1}\xi) & e_1^1(L_\mathcal{M}a)[BM^{-1}(\xi+\gamma_1)] & \cdots & e_1^{m-1}(L_\mathcal{M}a)[BM^{-1}(\xi+\gamma_{m-1})] \\
                       \vdots & \vdots & \vdots & \vdots \\
                       e_{m-1}^0(L_\mathcal{M}a)^{m-1}(BM^{-1}\xi) & e_{m-1}^1(L_\mathcal{M}a)^{m-1}[BM^{-1}(\xi+\gamma_1)] & \cdots
                             & e_{m-1}^{m-1}(L_\mathcal{M}a)^{m-1}[BM^{-1}(\xi+\gamma_{m-1})]\\
                     \end{array}
                   \right]}.
\end{equation}
Any $c\in l^2(\mathbb{Z}^2)$ can be recovered in a stable way, i.e., the inverse of $\mathcal{A}_{\mathcal{M}}(\xi)$ is bounded, from the dynamical sampling
measurements $y_0(k)=c(M^Tk), y_j(k)=(a^j\star_{d} c)(M^Tk), j=1,2,\cdots,m-1,\; k\in\mathbb{Z}^2$, if and only if there exists $\alpha>0$ so that the set $\{\xi\; |\,|\det(\mathcal{A}_{\mathcal{M}}(\xi))|<\alpha\}$ has zero measure.
\end{Theorem}

\begin{proof}
Combining Lemma~\ref{lem:D-Conv} and Lemma~\ref{lem:Poi}, for any $j=1,2,\cdots,m-1$, we get
\begin{align}
&\frac{m}{\sqrt{\det(iB)}}\{\mathcal{F}[y_jS_{M}(\lambda_{\mathcal{M}})]\}\big(\xi\big)\nonumber\\
=&\frac{m}{\sqrt{\det(iB)}}\{\mathcal{F}[S_M((a^j\star_{d} c)\lambda_{\mathcal{M}})]\}\big(\xi\big)\nonumber\\
=&\sum_{k=0}^{m-1}\overline{\eta}_\mathcal{M}[BM^{-1}(\xi+\gamma_k)](L_{\mathcal{M}}(a^j\star_{d} c))[BM^{-1}(\xi+\gamma_k)] \label{cao}\\
=&\sum_{k=0}^{m-1}\overline{\eta}_\mathcal{M}^2[BM^{-1}(\xi+\gamma_k)](L_{\mathcal{M}}(a^j))[BM^{-1}(\xi+\gamma_k)](L_{\mathcal{M}}c)[BM^{-1}(\xi+\gamma_k)]\nonumber\\
=&\sum_{k=0}^{m-1}\overline{\eta}_\mathcal{M}^{j+1}[BM^{-1}(\xi+\gamma_k)](L_{\mathcal{M}}a)^{j}[BM^{-1}(\xi+\gamma_k)](L_{\mathcal{M}}c)[BM^{-1}(\xi+\gamma_k)].\nonumber
\end{align}
Let
\begin{equation}\label{eq:dyna1}
\mathcal{Y}(\xi)=\left[
                     \begin{array}{cccc}
                       \frac{m}{\sqrt{\det(iB)}}\{\mathcal{F}[y_0S_{M}(\lambda_{\mathcal{M}})]\}\big(\xi\big)\\
                       \frac{m}{\sqrt{\det(iB)}}\{\mathcal{F}[y_1S_{M}(\lambda_{\mathcal{M}})]\}\big(\xi\big)\\
                       \vdots  \\
                       \frac{m}{\sqrt{\det(iB)}}\{\mathcal{F}[y_{m-1}S_{M}(\lambda_{\mathcal{M}})]\}\big(\xi\big) \\
                     \end{array}
                   \right],
\end{equation}
and
\[
\mathbf{\mathcal{C}}(\xi)=\left[
         \begin{array}{c}
           (L_\mathcal{M}c)(BM^{-1}\xi) \\
           (L_\mathcal{M}c)[BM^{-1}(\xi+\gamma_1)] \\
           \vdots \\
           (L_\mathcal{M}c)[BM^{-1}(\xi+\gamma_{m-1})] \\
         \end{array}
       \right].
\]
In short notation, from (\ref{cao}), we can equivalently rewrite (\ref{eq:dyna1}) as
\begin{equation}\label{eq:dyna2}
\mathcal{Y}(\xi)=\mathcal{A}_{\mathcal{M}}(\xi)\mathcal{C}(\xi).
\end{equation}
Thus, we can solve this equation (\ref{eq:dyna2}) with respect to $\mathcal{C}(\xi)$ (which we use to produce $c$), if $\mathcal{A}_{\mathcal{M}}(\xi)$ is invertible.
\end{proof}

\section{Multivariate dynamical sampling in shift-invariant spaces associated with 2D-NS-LCT}\label{sec:shift}

In this section, we naturally study the multivariate dynamical sampling in shift-invariant spaces associated with the 2D-NS-LCT. Shift-invariant spaces are typical spaces of functions considered in sampling theory.

First, we derive a sufficient and necessary condition under which a function $\phi(t)\in L^2(\mathbb{R}^2)$ can be a generator for a shift-invariant space $V$ associated with the 2D-NS-LCT.

\begin{Theorem}\label{thm:Riesz}
Let a sequence $s\in l^2(\mathbb{Z}^2)$ and a function $\phi\in L^2(\mathbb{R}^2)$. Assume that the chirp-modulated subspace of $L^2(\mathbb{R}^2)$ is given by
\[
V(\phi)=\overline{\{f\in L^2(\mathbb{R}^2): f(t)=(s\star_{sd}\phi)(t)\}}.
\]
Then, $\{e^{-2i\pi(t-k)^TB^{-1}(t-k)}\phi(t-k)\}$ is a Riesz basis for $V(\phi)$, if and only if there exist two constants $\eta_1,\eta_2>0$ such that
\begin{equation}
 \eta_1\le\sum_{k\in \mathbb{Z}^2}\big|(L_\mathcal{M}(\phi))(t+Bk)\big|^2\le \eta_2
\end{equation}
for all $t\in \left\{Bx \;| \,x\in \mathbb{T}^2\right\}$.
\end{Theorem}
\begin{proof}
For $f(t)=(s\star_{sd}\phi)(t)$, by Lemma~\ref{lem:Se-Conv}, we have
\[
(L_\mathcal{M}f)(\xi)=e^{-i\pi\xi^TDB^{-1}\xi}(L_\mathcal{M}s)(\xi)(L_\mathcal{M}\phi)(\xi).
\]
Thus,
\[
|(L_\mathcal{M}f)(\xi)|^2=|(L_\mathcal{M}s)(\xi)|^2|(L_\mathcal{M}\phi)(\xi)|^2.
\]
Since $|(L_\mathcal{M}s)(\xi)|$ is periodic with periodicity matrix $B$, we have, from (\ref{ssheng}),
\begin{align}
\|(L_\mathcal{M}f)(\xi)\|_{L^2(\mathbb{R}^2)}^2
&=\int_{\mathbb{R}^2}|(L_\mathcal{M}s)(\xi)|^2|(L_\mathcal{M}\phi)(\xi)|^2{\rm{d}}\xi\nonumber\\
&=\sum_{k\in \mathbb{Z}^2}\int_{\left\{B(t+k) \;| \,t\in \mathbb{T}^2\right\}}|(L_\mathcal{M}s)(\xi)|^2|(L_\mathcal{M}\phi)(\xi)|^2{\rm{d}}\xi\nonumber\\
&=\sum_{k\in \mathbb{Z}^2}\int_{\left\{Bt \;| \,t\in \mathbb{T}^2\right\}}|(L_\mathcal{M}s)(\xi+Bk)|^2|(L_\mathcal{M}\phi)(\xi+Bk)|^2{\rm{d}}\xi\nonumber\\
&=\int_{\left\{Bt \;| \,t\in \mathbb{T}^2\right\}}|(L_\mathcal{M}s)(\xi)|^2\sum_{k\in \mathbb{Z}^2}|(L_\mathcal{M}\phi)(\xi+Bk)|^2{\rm{d}}\xi\nonumber\\
&:=\int_{\left\{Bt \;| \,t\in \mathbb{T}^2\right\}}|(L_\mathcal{M}s)(\xi)|^2G_{\phi,\mathcal{M}}(\xi){\rm{d}}\xi,\label{nimei}
\end{align}
where $G_{\phi,\mathcal{M}}(\xi)=\sum_{k\in \mathbb{Z}^2}|(L_\mathcal{M}\phi)(\xi+Bk)|^2$ is the Grammian of $\phi$ associated with the 2D-NS-LCT. Notice that
\begin{align}
&\int_{\left\{Bt \;| \,t\in \mathbb{T}^2\right\}}|(L_\mathcal{M}s)(\xi)|^2{\rm{d}}\xi\nonumber\\
=&\frac{1}{\det{(B)}}\sum_{n\in \mathbb{Z}^2}\sum_{k\in \mathbb{Z}^2}s(n)\overline{s(k)}e^{i\pi n^TB^{-1}An-i\pi k^TB^{-1}Ak}\int_{\left\{Bt \;| \,t\in \mathbb{T}^2\right\}}
e^{-2i\pi (n-k)^{T}B^{-1}\xi}{\rm{d}}\xi\nonumber\\
=&\sum_{n\in \mathbb{Z}^2}\sum_{k\in \mathbb{Z}^2}s(n)\overline{s(k)}e^{i\pi n^TB^{-1}An-i\pi k^TB^{-1}Ak}\delta_{n,k}\nonumber\\
=&\sum_{n\in \mathbb{Z}^2}|s(n)|^2=\|s\|_{l^2}^2.
\end{align}
From (\ref{nimei}), we know that
\[
0<\eta_1\le G_{\phi,\mathcal{M}}(\xi)\le \eta_2<+\infty\, \mbox{ for all }\xi\in \left\{Bx \;| \,x\in \mathbb{T}^2\right\}
\]
is equivalent to
\[
\eta_1\|L_\mathcal{M}s\|^2=\eta_1\|s[k]\|_{l^2}^2\le\|(L_\mathcal{M}f)(\xi)\|_{L^2(\mathbb{R}^2)}^2\le \eta_2\|s[k]\|_{l^2}^2\le \eta_2\|L_\mathcal{M}s\|^2.
\]
This completes the proof.
\end{proof}

The local behavior and global decay of $\phi$ can be described in terms of the Wiener amalgam spaces as follows. A measurable function $f$ belongs to the Wiener
amalgam space $W(L^p(\mathbb{R}^2)),1\le p<\infty$, if it satisfies
\begin{equation}
\|f\|_{W(L^p(\mathbb{R}^2))}^p:=\sum_{k\in \mathbb{Z}^2}{\rm{ess}}\sup\{|f(x+k)|^p; x\in \mathbb{T}^2\}<+\infty.
\end{equation}

If $p=\infty$, a measurable function $f$ belongs to $W(L^\infty(\mathbb{R}^2))=L^\infty(\mathbb{R}^2)$, if it satisfies
\begin{equation}
\|f\|_{W(L^\infty(\mathbb{R}^2))}^p:=\sup_{k\in \mathbb{Z}^2}\{{\rm{ess}\sup}\{|f(x+k)|; x\in\mathbb{T}^2\}\}<+\infty.
\end{equation}

Considering that ideal sampling makes sense only for continuous functions, we therefore focus on the amalgam space $W_0(L^p(\mathbb{R}^2)):=W(L^p(\mathbb{R}^2))\cap
C(\mathbb{R}^2)$, where $C(\mathbb{R}^2)$ denotes the space of continuous functions on $\mathbb{R}^2$.

The multivariate dynamical sampling problem in shift-invariant spaces is to recover a function $f\in V(\phi)$ from its dynamical sampling measurements associated with the 2D-NS-LCT, i.e.,
\[
\left\{(a^j\star_c f)(M^Tk): j=1,2,\cdots,m-1,\;k\in \mathbb{Z}^2\right\},
\]
where $a^j=\underbrace{a\star_c\cdots\star_{c} a}_{j}$ and $a\in W(L^1(\mathbb{R}^2))$.

First, we present an important lemma, which will be used later.
\begin{Lemma}\label{lem:C1}
Let a sequence $s\in l^2(\mathbb{Z}^2)$ and two functions $f, g\in L^2(\mathbb{R}^2)$. Then, we have
\begin{equation}\label{eq:C1}
f\star_c(s\star_{sd} g)=s\star_{sd}(f\star_c g)
\end{equation}
\end{Lemma}

\begin{proof}
Letting $h_1(t)=(s\star_{sd}g)(t)$ and $h_2(t)=(f\star_{c}g)(t)$, we have
\begin{align*}
&(f\star_c h_1)(t)
=\frac{1}{\sqrt{\det{(iB)}}}e^{-i\pi t^{T}B^{-1}At}(\overrightarrow{f}\ast\overrightarrow{h}_1)(t)\\
=&\frac{1}{\sqrt{\det{(iB)}}}e^{-i\pi t^{T}B^{-1}At}\int_{\mathbb{R}^2}f(t-x)e^{i\pi(t-x)^{T}B^{-1}A(t-x)}h_1(x)e^{i\pi x^{T}B^{-1}Ax}{\rm{d}}x\nonumber\\
=&\frac{1}{\det{(iB)}}e^{-i\pi t^{T}B^{-1}At}\int_{\mathbb{R}^2}f(t-x)e^{i\pi(t-x)^{T}B^{-1}A(t-x)}\\
\quad &\times \bigg(\sum_{k\in \mathbb{Z}^2}s(k)e^{i\pi k^{T}B^{-1}Ak}g(x-k)e^{i\pi (x-k)^{T}B^{-1}A(x-k)}\bigg){\rm{d}}x\\
=&\frac{1}{\det{(iB)}}e^{-i\pi t^{T}B^{-1}At}\sum_{k\in \mathbb{Z}^2}s(k)e^{i\pi k^{T}B^{-1}Ak}\bigg(\int_{\mathbb{R}^2}f(t-k-y)\\
\quad &\times e^{i\pi(t-k-y)^{T}B^{-1}A(t-k-y)}g(y)e^{i\pi y^{T}B^{-1}Ay}{\rm{d}}y\bigg)\\
=&\frac{1}{\sqrt{\det{(iB)}}}e^{-i\pi t^{T}B^{-1}At}\sum_{k\in \mathbb{Z}^2}s(k)e^{i\pi k^{T}B^{-1}Ak}h_2(t-k)e^{i\pi (t-k)^{T}B^{-1}A(t-k)}\nonumber\\
=&(s\star_{sd} h_2)(t).
\end{align*}
This completes the proof.
\end{proof}

In the following, we propose a sufficient
and necessary condition for stably recovering $f$ from its dynamical sampling measurements $f(M^Tk), (a^j\star_c f)(M^Tk), j=1,2,\cdots,m-1, k\in \mathbb{Z}^2$ associated with the 2D-NS-LCT.

\begin{Theorem}
Let $\phi=\phi_0\in W_0(L^1(\mathbb{R}^2)), a\in W(L^1(\mathbb{R}^2))$, and $\phi_j=a^j\star_c \phi$, then $L_\mathcal{M}\phi_l^j\in C(\mathbb{R}^2)$ for $l, j=0,1,\cdots,m-1$, where $\phi_l^j(r)=\phi_j(M^Tr-\eta_l)e^{i\pi (M^Tr-\eta_l)^{T}B^{-1}A(M^Tr-\eta_l)}$. Any $f\in V(\phi)$ can be recovered in a stable way, i.e., the inverse of $\mathcal{B}(\xi)$ is bounded, from the dynamical sampling measurements $f(M^Tk), a^j\star_{c}f(M^Tk), j=1,2,\cdots, m-1,
\;k\in \mathbb{Z}^2$, if and only if $\det(\mathcal{B}(\xi))\ne 0$ for any $\xi\in \mathbb{R}^2$, where $\mathcal{B}(\xi)$ is defined by
\[
\mathcal{B}(\xi)=\left[
                   \begin{array}{cccc}
                     (L_\mathcal{M}\phi_0^0)(\xi) & (L_\mathcal{M}\phi_1^0)(\xi) & \cdots & (L_\mathcal{M}\phi_{m-1}^0)(\xi) \\
                     (L_\mathcal{M}\phi_0^1)(\xi) & (L_\mathcal{M}\phi_1^1)(\xi) & \cdots & (L_\mathcal{M}\phi_{m-1}^1)(\xi)\\
                     \vdots & \vdots & \vdots & \vdots \\
                     (L_\mathcal{M}\phi_0^{m-1})(\xi) & (L_\mathcal{M}\phi_1^{m-1})(\xi) & \cdots & (L_\mathcal{M}\phi_{m-1}^{m-1})(\xi)\\
                   \end{array}
                 \right].
\]
\end{Theorem}

\begin{proof}
Given a function $f\in V(\phi)$, we have
\[
f=(s\star_{sd}\phi)(t).
\]
By Lemma~\ref{lem:C1}, we obtain
\begin{align*}
v_j(k)\triangleq&(a^j\star_cf)(M^Tk)e^{i\pi(M^Tk)^TB^{-1}AM^Tk-i\pi k^TB^{-1}Ak}\\
      =&[a^j\star_c(s\star_{sd}\phi)](M^Tk)e^{i\pi(M^Tk)^TB^{-1}AM^Tk-i\pi k^TB^{-1}Ak}\\
      =&[s\star_{sd}(a^j\star_c \phi)](M^Tk)e^{i\pi(M^Tk)^TB^{-1}A(M^Tk)-i\pi k^TB^{-1}Ak}\\
      =&\frac{1}{\sqrt{\det{(iB)}}}e^{-i\pi k^{T}B^{-1}Ak}\sum_{n\in \mathbb{Z}^2}s(n)e^{i\pi n^{T}B^{-1}An}(a^j\star_c \phi)(M^Tk-n)\\
      &\quad\times e^{i\pi (M^Tk-n)^{T}B^{-1}A(M^Tk-n)}\\
      :=&\frac{1}{\sqrt{\det{(iB)}}}e^{-i\pi k^{T}B^{-1}Ak}\sum_{n\in \mathbb{Z}^2}s(n)e^{i\pi n^{T}B^{-1}An}\phi_j(M^Tk-n)\\
      &\quad\times e^{i\pi (M^Tk-n)^{T}B^{-1}A(M^Tk-n)}\\
      =&\frac{1}{\sqrt{\det{(iB)}}}e^{-i\pi k^{T}B^{-1}Ak}\sum_{l=0}^{m-1}\sum_{r\in\mathbb{Z}^2}s( M^Tr+\eta_l)e^{i\pi (M^Tr+\eta_l)^{T}B^{-1}A( M^Tr+\eta_l)}\\
      &\quad\times \phi_j(M^Tk-M^Tr-\eta_l)e^{i\pi (M^Tk-M^Tr-\eta_l)^{T}B^{-1}A(M^Tk-M^Tr-\eta_l)}\\
     =&\sum_{l=0}^{m-1}(s_l\star_{sd}\phi_l^j)(k),
\end{align*}
where $s_l(r)=s(M^Tr+\eta_l)e^{i\pi (M^Tr+\eta_l)^{T}B^{-1}A( M^Tr+\eta_l)}$.

Then,
by Lemma~\ref{lem:Se-Conv}, we readily have, for $j=1,2,\cdots,m-1$,
\begin{equation}\label{hdsda}
(L_{\mathcal{M}}v_j)(\xi)=\sum_{l=0}^{m-1}e^{-i\pi\xi^TDB^{-1}\xi}(L_\mathcal{M}s_l)(\xi)(L_\mathcal{M}\phi_l^j)(\xi).
\end{equation}
Let
\begin{equation*}
(L_{\mathcal{M}}v)(\xi)=\left[
                             \begin{array}{c}
                              (L_{\mathcal{M}}v_0)(\xi) \\
                              (L_{\mathcal{M}}v_1)(\xi) \\
                              \vdots\\
                              (L_{\mathcal{M}}v_{m-1})(\xi) \\
                              \end{array}
                              \right]
\end{equation*}
and
\begin{equation*}
(L_{\mathcal{M}}s)(\xi)=e^{-i\pi\xi^TD^{-1}B\xi}\left[
                                                     \begin{array}{c}
                                                      (L_{\mathcal{M}}s_0)(\xi) \\
                                                      (L_{\mathcal{M}}s_1)(\xi) \\
                                                      \vdots\\
                                                      (L_{\mathcal{M}}s_{m-1})(\xi) \\
                                                      \end{array}
                                                    \right].
\end{equation*}
Hence, from (\ref{hdsda}), we get
\begin{equation}\label{shi-2}
(L_{\mathcal{M}}v)(\xi)=\mathcal{B}(\xi)(L_{\mathcal{M}}s)(\xi).
\end{equation}
That is to say, we can solve the equation (\ref{shi-2}) with respect to $(L_{\mathcal{M}}s)(\xi)$, if $\mathcal{B}(\xi)$ is invertible.
\end{proof}

\section{Example of multivariate dynamical sampling}\label{jjjia}
In this section, to make the main results obtained above more transparent and more complete, we give a simple example to show that the necessary and sufficient condition in Theorem \ref{thm:dynam} is feasible.

Let matrix parameters $A, B, C, D$ of the 2D-DT-NS-LCT be
\[
A=\left[
  \begin{array}{cc}
    1 & 0 \\
    0 & 1 \\
  \end{array}
\right],\,
B=\left[
  \begin{array}{cc}
    1 & 1 \\
    1 & 3 \\
  \end{array}
\right],\,
C=\left[
  \begin{array}{cc}
    -0.5 & 0.5\\
    0.5 & 0.5 \\
  \end{array}
\right],\,
D=B,
\]
which obviously satisfy (\ref{ssads}). Let the dilation matrix $M\in\mathbb{Z}^{2\times2}$ be given by $M=B$. Then, we have $m=|\det(M)|=2$, and $\gamma_0=[0, 0]^T, \gamma_1=[1, 2]^T$. A sequence $a$ is given by
\[
a(k)=\begin{cases}
    c_1, & k=[-1, -1]^T,\\
    c_2, & k=[-1, -2]^T,\\
    0,  & \mbox{otherwise},
\end{cases}
\]
where $c_1, c_2\in\mathbb{R}$ and $c_2\neq0$. Let $\xi=[\xi_1, \xi_2]^T$, then we have
\[
\begin{split}
e_{0}^{0}&=\overline{\eta}_{\mathcal{M}}(\xi)=e^{-i\pi(\xi_1^2+\xi_2^2)},\\
e_{0}^{1}&=\overline{\eta}_{\mathcal{M}}(\xi+\gamma_1)=e^{-i\pi[(\xi_1+1)^2+(\xi_2+2)^2]},\\
e_{1}^{0}&=\overline{\eta}_{\mathcal{M}}^{2}(\xi)=e^{-2i\pi(\xi_1^2+\xi_2^2)},\\
e_{1}^{1}&=\overline{\eta}_{\mathcal{M}}^{2}(\xi+\gamma_1)=e^{-2i\pi[(\xi_1+1)^2+(\xi_2+2)^2]},\\
(L_\mathcal{M}a)(\xi)&=-\frac{c_1}{\sqrt{2}i}e^{i\pi(2\xi_1+\xi_1^2+\xi_2^2)}-\frac{c_2}{\sqrt{2}}e^{i\pi(\xi_1+\xi_2+\xi_1^2+\xi_2^2)},\\
(L_\mathcal{M}a)(\xi+\gamma_1)&=-\frac{c_1}{\sqrt{2}i}e^{i\pi[2\xi_1+(\xi_1+1)^2+(\xi_2+2)^2]}+\frac{c_2}{\sqrt{2}}e^{i\pi[\xi_1+\xi_2+(\xi_1+1)^2+(\xi_2+2)^2]}.
\end{split}
\]
Therefore, through easy computation, we get
\[
\begin{split}
\left|\det(\mathcal{A}_{\mathcal{M}}(\xi))\right|&=
\left|\det\left(\left[
                     \begin{array}{cc}
                       e_0^0 & e_0^1 \\
                       e_1^0(L_\mathcal{M}a)(\xi) & e_1^1(L_\mathcal{M}a)(\xi+\gamma_1)\\
                     \end{array}
                   \right]\right)\right|\\
                   &=\left|\sqrt{2}c_2e^{-i\pi[(\xi_1+1)^2+(\xi_2+2)^2-\xi_1-\xi_2+\xi_1^2+\xi_2^2]}\right|\\
                   &=\sqrt{2}\left|c_2\right|>0.
\end{split}
\]
That is to say, the sequence $a$ we choose can make $\mathcal{A}_{\mathcal{M}}(\xi)$ invertible for any $\xi$, thereby Theorem \ref{thm:dynam} holds.

\section{Conclusion}\label{sec5}
In this paper, we investigate the multivariate dynamical sampling in the $l^2(\mathbb{Z}^2)$ space associated with the 2D-DT-NS-LCT and shift-invariant spaces associated with the 2D-NS-LCT, respectively. More specifically, we obtain
a sufficient and necessary condition under which a sequence in $l^2(\mathbb{Z}^2)$ (or a function in a shift-invariant space $V(\phi)$ that is generated by $\phi\in L^2(\mathbb{R}^2)$) can be stably recovered from its dynamical sampling measurements associated with the 2D-DT-NS-LCT (or the 2D-NS-LCT). Our results extend the original ones in the FT domain.

%

\end{document}